\documentclass[oneside,english]{amsart}
\usepackage[T1]{fontenc}
\usepackage[latin9]{inputenc}
\usepackage{amsthm}
\usepackage{amstext}
\usepackage{amssymb}
\usepackage{esint}
\usepackage{graphicx}
\usepackage{enumerate}

\usepackage{graphicx}
\graphicspath{{noiseimages/}}

\makeatletter
\newtheorem{theorem}{Theorem}[section]
\newtheorem{lemma}[theorem]{Lemma}

\newtheorem{corollary}[theorem]{Corollary}
\numberwithin{figure}{section}
\theoremstyle{definition}

\newtheorem{example}[theorem]{Example}

\theoremstyle{remark}
\newtheorem{remark}[theorem]{Remark}

\numberwithin{equation}{section}
\makeatother

\usepackage{babel}

	\DeclareMathOperator{\dist}{dist}
	\DeclareMathOperator{\loc}{loc}
	
	\DeclareMathOperator*{\esssup}{ess\,sup}

	\DeclareMathOperator{\Imag}{Im}

\begin{document}

\title[Quasiconformal Mappings and Neumann Eigenvalues]{Quasiconformal Mappings and Neumann Eigenvalues of Divergent Elliptic Operators}

\author{V.~Gol'dshtein, V.~Pchelintsev, A.~Ukhlov}

\begin{abstract}
We study spectral properties of divergence form elliptic operators $-\textrm{div} [A(z) \nabla f(z)]$ with the Neumann boundary condition in planar domains (including some fractal type domains), that satisfy to the quasihyperbolic boundary conditions. Our method is based on an interplay between quasiconformal mappings, elliptic operators and composition operators on Sobolev spaces.
\end{abstract}
\maketitle

\section{Introduction}

In this paper we apply methods of the (quasi)conformal geometry to spectral problems for $A$-divergent form elliptic operators
with the Neumann boundary condition
\begin{equation}\label{EllDivOper}
L_{A}=-\textrm{div} [A(z) \nabla f(z)], \quad z=(x,y)\in \Omega, \quad
\left\langle A(z) \nabla f, n \right\rangle\big|_{\partial \Omega}=0,
\end{equation}
in the large class of (non)convex domains $\Omega \subset \mathbb C$ that satisfy the quasihyperbolic boundary condition \cite{KOT01,KOT02}. Here matrix functions $A(z)=\left\{a_{kl}(z)\right\}$  with measurable entries $a_{kl}(z)$ belongs to a class  $M^{2 \times 2}(\Omega)$ of all $2 \times 2$ symmetric matrix functions that satisfy to an additional condition $\textrm{det} A=1$ a.e. and to the uniform ellipticity condition:
\begin{equation}\label{UEC}
\frac{1}{K}|\xi|^2 \leq \left\langle A(z) \xi, \xi \right\rangle \leq K |\xi|^2 \,\,\, \text{a.e. in}\,\,\, \Omega,
\end{equation}
for every $\xi \in \mathbb C$ and for some $1\leq K< \infty$.

Such type of elliptic operators arise in various problems of mathematical physics (see, for example, \cite{AIM}).

The suggested method is based on a (quasi)conformal representation of a non smooth Riemannian metric in the domain $\Omega$:
\[
ds^2=a_{11}(x,y)dx^2+2a_{12}(x,y)dxdy+a_{22}(x,y)dy^2
\]
induced by the matrix $A$. The complex dilatation $\mu$ of the corresponding quasiconformal mapping $\varphi$ from $\Omega$ to the unit disc $\mathbb D\subset\mathbb C$ can be calculated using the matrix $A$ (see, for example, \cite[p. 412]{AIM}). Inverse, if the complex dilatation $\mu$ is given then the matrix $A$ can be reproduced. It means that there is one to one correspondence between the matrices and the complex dilatations. By the construction the quasiconformal mapping $\varphi:\Omega\to\mathbb D$ is an isometry of the domain $\Omega$ with this new Riemannian metric $ds$ (induced by the matrix $A$) and the unit disc $\mathbb{D}$ with the hyperbolic metric. It is reasonable to call such metric $ds$ as an $A$-quasiconformal metric and the quasiconformal mapping $\varphi$ as an $A$-quasiconformal mapping (i.e. quasiconformal mapping agreed with the matrix $A$).

Hence, we can conclude that an $A$-quasiconformal mapping is conformal (i.e. preserve angles) in this $A$-quasiconformal Riemannian metric.

Let us remind that conformal homeomorphisms induce isometries of uniform Sobolev spaces $L^{1,2}(\Omega)$ and $L^{1,2}(\mathbb D)$. In the present article we prove that  $A$-quasi\-con\-for\-mal mappings induce isometries of an uniform Sobolev space $L_A^{1,2}(\Omega)$ (Sobolev space agreed with the matrix $A$) and the uniform Sobolev space $L^{1,2}(\mathbb D)$.

\vskip 0.2cm
{\bf Conjecture.} {\it Spectral properties of the $A$-divergent form elliptic operators with the Neumann boundary condition depends on
$A$-quasiconformal geometry of domains only.}
\vskip 0.2cm

The suggested method is based on connections between composition operators on Sobolev spaces, elliptic operators and quasiconformal mappings. We prove that $\varphi:\Omega\to\Omega'$ is a quasiconformal mapping  agreed with the matrix $A$ (i.e. induced by $A$ via the Beltrami equation)  if and only if
\[
\iint\limits_\Omega \left\langle A(z)\nabla f(\varphi(z)),\nabla f(\varphi(z))\right\rangle\,dxdy=\iint\limits_{\Omega'} \left\langle \nabla f(w),\nabla f(w)\right\rangle\,dudv,
\]
for all $f\in L^{1,2}(\Omega')$.

This result permits us to introduce an $A$-norm in the corresponding uniform Sobolev space $L^{1,2}_A$. For this norm the $A$-quasiconformal mappings play a role similar to conformal mappings for the Laplace operator and uniform Sobolev spaces $L^{1,2}$. In particular, we prove that the $A$-quasiconformal mappings generalize the well known property of conformal mappings to generate isometries of uniform Sobolev spaces $L^{1,2}(\Omega)$ and $L^{1,2}(\Omega')$ (see, for example, \cite{C50}). In terms of the $A$-quasiconformal mappings we also refine the functional characterization of quasiconformal mappings, obtained in the article \cite{VG75} in the terms of isomorphisms of uniform Sobolev spaces $L^{1,2}$.

Short historical remarks. Spectral estimates of elliptic operators eigenvalues represent an important part of the modern spectral theory (see, for example,
\cite{A98,AB06,BLL,BCT15,EP15,ENT,FNT,LM98}).
The classical upper estimate for the first non-trivial Neumann eigenvalue of the Laplace operator (the $A$-divergent form elliptic operator with the matrix $A=I$)
\begin{equation*}
\mu_1(I,\Omega):=\mu_1(\Omega)\leq \mu_1(\Omega^{\ast})=\frac{{j_{1,1}'^2}}{R^2_{\ast}}
\end{equation*}
was proved by Szeg\"o \cite{S54} for simply connected planar domains via a conformal mappings technique ("the method of conformal normalization"). In this inequality $j_{1,1}'$ denotes the first positive zero of the derivative of the Bessel function $J_1$ and $\Omega^{\ast}$ is a disc of the same area as $\Omega$ with $R_{\ast}$ as its radius.

In convex domains $\Omega\subset\mathbb R^n$, $n\geq 2$,  the classical lower estimates of the Neumann eigenvalues of the Laplace operator \cite{PW} state that
\begin{equation}
\label{PW}
\mu_1(\Omega)\geq \frac{\pi^2}{d(\Omega)^2},
\end{equation}
where $d(\Omega)$ is a diameter of a convex domain $\Omega$. Similar estimates for the non-linear $p$-Laplace operator, $p\ne 2$, were obtained much later in \cite{ENT}.

Unfortunately, for non-convex domains $\mu_1(\Omega)$ can not be characterized in the terms of its Euclidean diameters. It can be seen by considering a domain consisting of two identical squares connected by a thin corridor \cite{BCDL16}.

Let us return to our studies. In the previous works  \cite{GPU17_2,GPU19,GU16} we returned to applications of a (quasi)conformal mappings techniques to such estimates in rough (non-convex) domains. Let us remind that some applications of a conformal mappings to this problem can be found in \cite{S54}). We used (quasi)conformal mappings in a framework of composition operators on Sobolev spaces \cite{U93,VG75,VU02}. This method permitted us to obtain lower estimates of the first non-trivial Neumann-Laplace eigenvalue $\mu_1(\Omega)$  in the terms of the hyperbolic (conformal) radius of $\Omega$ for a large class of domains that includes some fractal domains.

In this paper we use the $A$-quasiconformal mappings via the composition operator theory. The corresponding composition operators (isometries for the norm induced by matrices $A$) allows us reduce the spectral problem for the divergence form elliptic operator \eqref{EllDivOper} defined in a simply connected domain $\Omega\subset\mathbb C$ to a weighted spectral problem for the Laplace operator in the unit disc $\mathbb D\subset\mathbb C$.

Roughly speaking, by the chain rule applied to a function $f(z)=g \circ \varphi(z)$ \cite{GNR18}, we have
\begin{equation}\label{QCE}
-\textrm{div} [A(z) \nabla f(z)] = -\textrm{div} [A(z) \nabla g(\varphi(z))]= -\left|J(w,\varphi^{-1})\right|^{-1} \Delta g(w),
\end{equation}
where the weight $J(w,\varphi^{-1})$ is the Jacobian of the inverse mapping $\varphi^{-1}:\mathbb D\to\Omega$.

As an example we consider the divergent form operator $-\textrm{div} [A(z) \nabla f(z)]$ with the matrix
$$
A(z)=\begin{pmatrix} \frac{a+b}{a-b} & 0 \\ 0 &  \frac{a-b}{a+b} \end{pmatrix},\,\,a>b\geq 0,
$$
defined in the interior of ellipse $\Omega_e$ with semi-axes $a+b$ and $a-b$.  By Theorem~\ref{T4.7} we have
$$
\mu_1(A,\Omega_e) \geq \frac{(j'_{1,1})^2}{a^2-b^2},
$$
what is better (Example~\ref{example1}) than the lower estimate obtained by using the classical estimate (\ref{PW}) and the uniform ellipticity condition:
$$
\mu_1(A,\Omega_e) \geq \frac{\pi^2}{4(a+b)^2} \frac{a-b}{a+b}.
$$
For thin ellipses, i.e $a+b$ fixed and $(a-b)$ tends to zero an asymptotic of our estimate is $\infty$ when the classical asymptotic is $0$.

%In the present paper we use connection between quasiconformal mappings and the $A$-divergent form elliptic operators.

The application of the composition operators theory to spectral problems of the $A$-divergent form elliptic operators is based on reducing of a positive quadratic form
\[
ds^2=a_{11}(x,y)dx^2+2a_{12}(x,y)dxdy+a_{22}(x,y)dy^2
\]
defined in a planar domain $\Omega$, by means of a quasiconformal change of variables, to the canonical form
\[
ds^2=\Lambda(du^2+dv^2),\,\, \Lambda\neq 0,\,\, \text{a.e. in}\,\, \Omega',
\]
given that $a_{11}a_{22}-a^2_{12} \geq \kappa_0>0$, $a_{11}>0$, almost everywhere in $\Omega$ \cite{Ahl66, AIM, BGMR}. Note that this fact can be extended to linear operators of the form $\textrm{div} [A(z) \nabla f(z)]$, $z=x+iy$, for matrix function $A \in M^{2 \times 2}(\Omega)$.

Let $\xi(z)=\Re(\varphi(z))$ be a real part of a quasiconformal mapping $\varphi(z)=\xi(z)+i \eta(z)$, which satisfies to the Beltrami equation:
\begin{equation}\label{BelEq}
\varphi_{\overline{z}}(z)=\mu(z) \varphi_{z}(z),\,\,\, \text{a.e. in}\,\,\, \Omega,
\end{equation}
where
$$
\varphi_{z}=\frac{1}{2}\left(\frac{\partial \varphi}{\partial x}-i\frac{\partial \varphi}{\partial y}\right) \quad \text{and} \quad
\varphi_{\overline{z}}=\frac{1}{2}\left(\frac{\partial \varphi}{\partial x}+i\frac{\partial \varphi}{\partial y}\right),
$$
with the complex dilatation $\mu(z)$ is given by
\begin{equation}\label{ComDil}
\mu(z)=\frac{a_{22}(z)-a_{11}(z)-2ia_{12}(z)}{\det(I+A(z))},\quad I= \begin{pmatrix} 1 & 0 \\ 0 & 1 \end{pmatrix}.
\end{equation}
We call this quasiconformal mapping (with the complex dilatation $\mu$ defined by (\ref{ComDil})) as an $A$-quasiconformal mapping.

Note that the uniform ellipticity condition \eqref{UEC} can be reformulated as
\begin{equation}\label{OVCE}
|\mu(z)|\leq \frac{K-1}{K+1},\,\,\, \text{a.e. in}\,\,\, \Omega.
\end{equation}

Conversely, using the complex dilatation $\mu$ we can obtain from \eqref{ComDil} (see, for example, \cite[p.412]{AIM}) the following representation of the matrix $A$ :
\begin{equation}\label{Matrix-F}
A(z)= \begin{pmatrix} \frac{|1-\mu|^2}{1-|\mu|^2} & \frac{-2 \Imag \mu}{1-|\mu|^2} \\ \frac{-2 \Imag \mu}{1-|\mu|^2} &  \frac{|1+\mu|^2}{1-|\mu|^2} \end{pmatrix},\,\,\, \text{a.e. in}\,\,\, \Omega.
\end{equation}

So, given any $A \in M^{2 \times 2}(\Omega)$, one produced, by \eqref{OVCE}, the complex dilatation $\mu(z)$, for which, in turn, the Beltrami equation \eqref{BelEq} induces a quasiconformal homeomorphism $\varphi:\Omega \to \varphi(\Omega)$ as its solution, by the Riemann measurable mapping theorem (see, for example, \cite{Ahl66}). We will say that the matrix function $A$ induces the corresponding $A$-quasiconformal homeomorphism $\varphi$ or that $A$ and $\varphi$ are agreed. The $A$-quasiconformal mapping $\psi: \Omega \to \mathbb D$ of simply connected domain $\Omega \subset \mathbb C$
onto the unit disc $\mathbb D \subset \mathbb C$ can be obtained as a composition of $A$-quasiconformal homeomorphism $\varphi:\Omega \to \varphi(\Omega)$ and a conformal mapping $\omega : \varphi(\Omega) \to \mathbb D$.

So, by the given an $A$-divergent form elliptic operator defined in a domain $\Omega\subset\mathbb C$ we construct an $A$-quasiconformal mapping $\psi: \Omega \to \mathbb D$ with a metric quasiconformality coefficient
$$
K=\frac{1+\|\mu\mid L^{\infty}(\Omega)\|}{1-\|\mu\mid L^{\infty}(\Omega)\|},
$$
where $\mu$ defined by (\ref{ComDil}).

We prove that any $A$-quasiconformal mapping $\varphi: \Omega \to \Omega'$ induces an isometry of the spaces $L^{1,2}_A(\Omega)$ and $L^{1,2}(\Omega')$. This is the main technical result of this paper about Sobolev spaces. Using applications of quasiconformal mappings to the Sobolev type embedding theorems \cite{GG94,GU09}, we prove discreteness of the spectrum of the divergence form elliptic operators $-\textrm{div} [A(z) \nabla f(z)]$ with the Neumann boundary condition. Well-known estimates of constants in the Sobolev-Poincar\'e inequality for the unit disc and the previous isometry result in the framework of the composition operator theory permit us to obtain lower estimates of Neumann eigenvalues in the terms of integrals of derivatives of $A$-quasiconformal mappings for a large class of rough domains that includes a subclass of domains with fractal boundaries (quasidiscs).

From geometrical point of view it means that we study a class domains $\Omega\subset\mathbb C$ equipped with the corresponding quasiconformal geometry. Any such domain can be considered as a Riemannian manifold and we suppose that our estimates of Neumann eigenvalues are closely connected to spectral estimates of the Beltrami-Laplace operator.

\section{Sobolev spaces and $A$-quasiconformal mappings}

Let $E \subset \mathbb C$ be a measurable set on the complex plane and $h:E \to \mathbb R$ be a positive a.e. locally integrable function i.e. a weight. The weighted Lebesgue space $L^p(E,h)$, $1\leq p<\infty$,
is the space of all locally integrable functions endowed with the following norm:
$$
\|f\,|\,L^{p}(E,h)\|= \left(\iint\limits_E|f(z)|^ph(z)\,dxdy \right)^{\frac{1}{p}}< \infty.
$$

The two-weighted Sobolev space $W^{1,p}(\Omega,h,1)$, $1\leq p< \infty$, is defined
as the normed space of all locally integrable weakly differentiable functions
$f:\Omega\to\mathbb{R}$ endowed with the following norm:
\[
\|f\mid W^{1,p}(\Omega,h,1)\|=\|f\,|\,L^{p}(\Omega,h)\|+\|\nabla f\mid L^{p}(\Omega)\|.
\]

In the case $h=1$ this weighted Sobolev space coincides with the classical Sobolev space $W^{1,p}(\Omega)$.

The seminormed Sobolev space $L^{1,p}(\Omega)$, $1\leq p< \infty$,
is the space of all locally integrable weakly differentiable functions $f:\Omega\to\mathbb{R}$ endowed
with the following seminorm:
\[
\|f\mid L^{1,p}(\Omega)\|=\|\nabla f\mid L^p(\Omega)\|, \,\, 1\leq p<\infty.
\]

We also need a weighted seminormed Sobolev space $L_{A}^{1,2}(\Omega)$ (associated with the matrix $A$), defined
as the space of all locally integrable weakly differentiable functions $f:\Omega\to\mathbb{R}$ endowed with the following norm:
\[
\|f\mid L_{A}^{1,2}(\Omega)\|=\left(\iint\limits_\Omega \left\langle A(z)\nabla f(z),\nabla f(z)\right\rangle\,dxdy \right)^{\frac{1}{2}}.
\]

The corresponding  Sobolev space $W^{1,2}_{A}(\Omega)$ is defined
as the normed space of all locally integrable weakly differentiable functions
$f:\Omega\to\mathbb{R}$ endowed with the following norm:
\[
\|f\mid W^{1,2}_{A}(\Omega)\|=\|f\,|\,L^{2}(\Omega)\|+\|f\mid L^{1,2}_{A}(\Omega)\|.
\]

These Sobolev spaces are closely connected with quasiconformal mappings.
Recall that a homeomorphism $\varphi: \Omega\to \Omega'$, where $\Omega,\, \Omega'\subset\mathbb C$, is called a $K$-quasiconformal mapping if $\varphi\in W^{1,2}_{\loc}(\Omega)$ and there exists a constant $1\leq K<\infty$ such that
$$
|D\varphi(z)|^2\leq K |J(z,\varphi)|\,\,\text{for almost all}\,\,z\in\Omega.
$$

An important subclass of quasiconformal mappings represent the class of bi-Lipschitz mappings. Note that a homeomorphism $\varphi: \Omega\to \Omega'$ is said to be an $L$-bi-Lipschitz if it satisfies the double inequality
\begin{equation}\label{Bi-Lip}
\frac{1}{L}|z-z'| \leq |\varphi(z)- \varphi(z')| \leq L|z-z'|,
\end{equation}
whenever $z, z' \in \Omega$. The smallest $L \geq 1$ for which \eqref{Bi-Lip} holds is called the isometric distortion of $\varphi$. It is known (see, for example, \cite{VGR}) that each $L$-bi-Lipschitz mapping $\varphi$ is $L^2$-quasiconformal.

Conversely we have the following connection between quasiconformal and bi-Lipschitz mappings:
\begin{lemma}
Let $\varphi: \Omega\to \Omega'$ be a $K$-quasiconformal mapping such that $|J(z, \varphi)|=1$ for almost all $z \in \Omega$. Then $\varphi$ is locally $\sqrt{K}$-bi-Lipschitz a.e. in $\Omega$.
\end{lemma}

\begin{proof}
Since $\varphi: \Omega\to \Omega'$ is a $K$-quasiconformal mapping then $\varphi$ is differentiable almost everywhere in $\Omega$ and we have
$$
|D\varphi(z)|^2\leq K |J(z,\varphi)|\,\,\,\text{for almost all}\,\,\,z\in\Omega.
$$
Because $|J(z, \varphi)|=1$ a.e. in $\Omega$ we obtain
$$
\lim\limits_{z'\to z}\frac{|\varphi(z)- \varphi(z')|}{|z-z'|}=|D\varphi(z)|\leq \sqrt{K}\,\,\, \text{for almost all}\,\,\, z\in\Omega.
$$
Hence, $\varphi$ is locally $L$-Lipschitz a.e. in $\Omega$ with $L \leq \sqrt{K}$.

On the other hand, it is known that the inverse mapping to $\varphi$ is again $K$-quasiconformal. So, $\varphi^{-1}$ is also locally $L$-Lipschitz a.e. in $\Omega$ with $L \leq \sqrt{K}$. Hence, $\varphi$ is locally $\sqrt{K}$-bi-Lipschitz a.e. in $\Omega$.

\end{proof}

Now we study a connection between composition operators on Sobolev spaces and the $A$-quasiconformal mappings that refine (in the case $n=2$) the corresponding assertion for quasiconformal mappings (\cite{VG75}).

\begin{theorem}\label{IsomSS}
Let $\Omega,\Omega'$ be domains in $\mathbb C$. Then a homeomorphism $\varphi :\Omega \to \Omega'$ is an $A$-quasiconformal mapping
if and only if $\varphi$ induces, by the composition rule $\varphi^{*}(f)=f \circ \varphi$,
an isometry of Sobolev spaces $L^{1,2}_A(\Omega)$ and $L^{1,2}(\Omega')$:
\[
\|\varphi^{*}(f)\,|\,L^{1,2}_A(\Omega)\|=\|f\,|\,L^{1,2}(\Omega')\|
\]
for any $f \in L^{1,2}(\Omega')$.
\end{theorem}

\begin{proof}
Sufficiency. We prove that if $\varphi :\Omega \to \Omega'$ is an $A$-quasiconformal mapping
then the composition operator
\[
\varphi^{*}:L^{1,2}(\Omega') \to L^{1,2}_A(\Omega),\,\,\, \varphi^{*}(f)=f \circ \varphi
\]
is an isometry.
Let $f \in L^{1,2}(\Omega')$ be a smooth function. Then the composition $g(z)=f \circ \varphi(z)$
is defined on $\Omega$ and is weakly differentiable almost everywhere in $\Omega$ \cite{VG75}. Let us check that $g(z)=f \circ \varphi(z)$ belongs to the Sobolev space $L^{1,2}_A(\Omega)$. By the chain rule \cite{VGR} we have
\begin{multline*}
\|g\,|\,L^{1,2}_A(\Omega)\| = \left(\iint\limits_{\Omega} \left\langle A(z) \nabla (f \circ \varphi(z)), \nabla (f \circ \varphi(z)) \right\rangle dxdy\right)^{\frac{1}{2}} \\
= \left(\iint\limits_{\Omega} |\nabla f|^{2}(\varphi(z)) |J(z,\varphi)| dxdy\right)^{\frac{1}{2}}
\\
=\left(\iint\limits_{\Omega'} |\nabla f|^{2}(w) dudv\right)^{\frac{1}{2}} =\|f\,|\,L^{1,2}(\Omega')\|.
\end{multline*}

Let $f \in L^{1,2}(\Omega')$ be an arbitrary function. Then there exists a sequence $\{f_k\}$, $k=1,2,...$ of smooth functions such that $f_k \in L^{1,2}(\Omega')$,
$$
\lim\limits_{k\to\infty}\|f-f_k\mid L^{1,2}(\Omega')\|=0
$$
and $\{f_k\}$ converges to $f$ a.e. in $\Omega'$.

Denote by $g_k=f_k\circ\varphi$, $k=1,2,...\,$.
Then
$$
\|g_k-g_l\,|\,L^{1,2}_A(\Omega)\|=\|f_k-f_l\,|\,L^{1,2}(\Omega')\|,\,\,k,l\in\mathbb N,
$$
and because the sequence $\{f_k\}$ converges in $L^{1,2}(\Omega')$ then the sequence $\{g_k\}$ converges in $L^{1,2}_A(\Omega)$.

Note that quasiconformal mappings possess the $N^{-1}$-Luzin property. It means that the preimage of a set of measure zero has measure zero.
So, the sequence $g_k=f_k\circ\varphi$ converges to $g=f\circ\varphi$ a.e. in $\Omega$ and hence in $L^{1,2}_A(\Omega)$.

Therefore
\[
\|\varphi^{*}(f)\,|\,L^{1,2}_A(\Omega)\|=\|f\,|\,L^{1,2}(\Omega')\|
\]
for any $f \in L^{1,2}(\Omega')$.

Necessity. Suppose that the composition operator
\[
\varphi^{*}:L^{1,2}(\Omega') \to L^{1,2}_A(\Omega)
\]
is an isometry, i.e.
\begin{equation}
\label{eqA}
\iint\limits_{\Omega} \left\langle A(z) \nabla (f \circ \varphi(z)), \nabla (f \circ \varphi(z)) \right\rangle dxdy
= \iint\limits_{\Omega'} |\nabla f|^{2}(w) dudv.
\end{equation}

Because the matrix $A$ satisfies the uniform ellipticity condition (\ref{UEC}) then by \eqref{eqA} we have
\begin{multline*}
%\label{eqAA}
\frac{1}{K}\iint\limits_{\Omega} |\nabla (f \circ \varphi(z))|^2~dxdy\leq
\iint\limits_{\Omega} \left\langle A(z) \nabla (f \circ \varphi(z)), \nabla (f \circ \varphi(z)) \right\rangle dxdy\\
= \iint\limits_{\Omega'} |\nabla f|^{2}(w) dudv.
\end{multline*}
Hence the following inequality
$$
\left(\iint\limits_{\Omega} |\nabla (f \circ \varphi(z))|^2~dxdy\right)^{\frac{1}{2}}
\leq K^{\frac{1}{2}} \left(\iint\limits_{\Omega'} |\nabla f|^{2}(w)~dudv\right)^{\frac{1}{2}}
$$
holds for any $f\in L^{1,2}(\Omega')$.

So, by \cite{VG75} we can conclude that the mapping $\varphi :\Omega \to \Omega'$ will be a $K$-quasiconformal mapping. Hence, by \cite{Ahl66} $\varphi$ will be a solution of the Beltrami equation
\begin{equation}\label{Beltr}
\varphi_{\overline{z}}(z)=\nu(z)\varphi_{z}(z),\,\,\, \text{a.e. in}\,\,\, \Omega
\end{equation}
with some complex dilatation $\nu(z)$, $|\nu(z)|<1$ a.e. in $\Omega$.

Now we consider the matrix $B$ generated by the complex dilatation $\nu(z)$:
\[
B(z)=\begin{pmatrix} \frac{|1-\nu|^2}{1-|\nu|^2} & \frac{-2 \Imag \nu}{1-|\nu|^2} \\ \frac{-2 \Imag \nu}{1-|\nu|^2} &  \frac{|1+ \nu|^2}{1-|\nu|^2} \end{pmatrix},\,\,\, \text{a.e. in}\,\,\, \Omega.
\]
Then $\varphi$ is a $B$-quasiconformal mapping. Because $\varphi$ defined by (\ref{Beltr}) we have finally
\begin{equation}
\label{eqB}
\iint\limits_{\Omega} \left\langle B(z) \nabla (f \circ \varphi(z)), \nabla (f \circ \varphi(z)) \right\rangle dxdy
= \iint\limits_{\Omega'} |\nabla f|^{2}(w) dudv
\end{equation}
for any $f \in L^{1,2}(\Omega')$.

Now using the equalities (\ref{eqA}) and (\ref{eqB}) we obtain
\[
\iint\limits_{\Omega} \left\langle A(z) \nabla g(z), \nabla g(z) \right\rangle dxdy =
\iint\limits_{\Omega} \left\langle B(z) \nabla g(z), \nabla g(z) \right\rangle dxdy
\]
for any $g \in L^{1,2}_{A}(\Omega)$.
It  means that Hilbert spaces $W^{1,2}_A(\Omega)$ and $W^{1,2}_B(\Omega)$ coincide. Therefore $A=B$ and $\mu=\nu$ a.e. in $\Omega$.
\end{proof}

Next, we set the following property for $A$-quasiconformal mappings.
\begin{lemma}
Let $\varphi : \Omega \to \mathbb D$ be an $A$-quasiconformal mapping. Then the inverse mapping $\psi=\varphi^{-1} : \mathbb D \to \Omega$ is $A^{-1}$-quasiconformal.
\end{lemma}

\begin{proof}
Let $\varphi : \Omega \to \mathbb D$ be an $A$-quasiconformal mapping with the matrix $A$ defined by the formula \eqref{Matrix-F}, i.e.
$$
A(z)= \begin{pmatrix} \frac{|1-\mu(z)|^2}{1-|\mu(z)|^2} & \frac{-2 \Imag \mu(z)}{1-|\mu(z)|^2} \\ \frac{-2 \Imag \mu(z)}{1-|\mu(z)|^2} &  \frac{|1+\mu(z)|^2}{1-|\mu(z)|^2} \end{pmatrix},\,\,\, \text{a.e. in}\,\,\, \Omega.
$$

By \cite{Ahl66} it is known that the complex dilatation for the inverse mapping $\varphi^{-1} : \mathbb D \to \Omega$ satisfies
\[
\mu_{\varphi^{-1}}(w)=-\nu_{\varphi} \circ \varphi^{-1}(w)\,\,\,\text{for almost all}\,\,\,w\in \mathbb D,
\]
where
\[
\nu_{\varphi}=\frac{\varphi_{\overline{z}}}{\overline{\varphi_z}}=\left(\frac{\varphi_z}{|\varphi_z|}\right)^2\mu_{\varphi},\,\,\, \text{a.e. in}\,\,\, \Omega
\]
is called the second complex dilatation of $\varphi$.

Hence, the matrix $B$ induces by the complex dilatation $\mu_{\varphi^{-1}}$ of the inverse mapping $\varphi^{-1}$ has the form
$$
B(w)= \begin{pmatrix} \frac{|1+\nu_{\varphi} \circ \varphi^{-1}|^2}{1-|\nu_{\varphi} \circ \varphi^{-1}|^2} & \frac{2 \Imag (\nu_{\varphi} \circ \varphi^{-1})}{1-|\nu_{\varphi} \circ \varphi^{-1}|^2} \\ \frac{2 \Imag (\nu_{\varphi} \circ \varphi^{-1})}{1-|\nu_{\varphi} \circ \varphi^{-1}|^2} &  \frac{|1-\nu_{\varphi} \circ \varphi^{-1}|^2}{1-|\nu_{\varphi} \circ \varphi^{-1}|^2} \end{pmatrix},
\,\,\,\text{a.e. in}\,\,\,\mathbb D.
$$

Because $\det B=1$, $|\mu_{\varphi}(z)|=|\mu_{\varphi^{-1}}(w)|$, $\Imag \mu_{\varphi} = -\Imag (\nu_{\varphi} \circ \varphi^{-1})$ for almost all $z\in \Omega$ and almost all $w=\varphi(z)\in \mathbb D$ we have
$$
A(z)B(\varphi(z))=I\,\,\,\text{for almost all}\,\,\, z\in \Omega.
$$

Therefore we conclude that $B(w)=A^{-1}(\varphi^{-1}(w))$ for almost all $w\in \mathbb D$ and $A^{-1}(z)=B(\varphi(z))$ for almost all $z\in \Omega$.
	
\end{proof}

\section{ Weighted Sobolev-Poincar\'e inequalities}

Denote by $B_{r,2}(\mathbb D)$, $1<r<\infty$,  the best constant in the (non-weighted) Sobolev-Poincar\'e inequality in the unit disc $\mathbb D$. Exact calculations of $B_{r,2}(\mathbb D)$, $r\ne 2$, is an open problem and we use the upper estimate (see, for example, \cite{GT77,GU16}):
$$
B_{r,2}(\mathbb D) \leq \left(2^{-1} \pi\right)^{\frac{2-r}{2r}}\left(r+2\right)^{\frac{r+2}{2r}}.
$$

On the basis of Theorem~\ref{IsomSS} we prove an universal weighted Sobolev-Poincar\'e inequality which holds in any simply connected planar domain with non-empty boundary. Denote by $h(z) =|J(z,\varphi)|$  the quasihyperbolic weight defined by an $A$-quasiconformal mapping $\varphi : \Omega \to \mathbb D$.

\begin{theorem}\label{Th4.1}
Let $A$ belongs to a class  $M^{2 \times 2}(\Omega)$  and $\Omega$ be a simply connected planar domain.
Then for any function $f \in W^{1,2}_{A}(\Omega)$ the following weighted Sobolev-Poincar\'e inequality
\[
\inf\limits_{c \in \mathbb R}\left(\iint\limits_\Omega |f(z)-c|^rh(z)dxdy\right)^{\frac{1}{r}} \leq B_{r,2}(h,A,\Omega)
\left(\iint\limits_\Omega \left\langle A(z) \nabla f(z), \nabla f(z) \right\rangle dxdy\right)^{\frac{1}{2}}
\]
holds for any $r \geq 1$ with the constant $B_{r,2}(h,A,\Omega) = B_{r,2}(\mathbb D)$.
\end{theorem}

\begin{proof}
Because $\Omega$ is a simply connected planar domain, then there exists \cite{Ahl66} an $\mu$-quasiconformal homeomorphism $\varphi : \Omega \to \mathbb D$ with
\begin{equation*}
\mu(z)=\frac{a_{22}(z)-a_{11}(z)-2ia_{12}(z)}{\det(I+A(z))},
\end{equation*}
which is an $A$-quasiconformal mapping.

Hence by Theorem~\ref{IsomSS} the equality
\begin{equation}\label{IN2.1}
||f \circ \varphi^{-1} \,|\, L^{1,2}(\mathbb D)|| = ||f \,|\, L^{1,2}_{A}(\Omega)||
\end{equation}
holds for any function $f \in L^{1,2}_{A}(\Omega)$.

Denote by $h(z):=|J(z,\varphi)|$ the quasihyperbolic weight in $\Omega$.
Now using the change of variable formula for the quasiconformal mappings \cite{VGR}, the equality \eqref{IN2.1} and the classical Sobolev-Poincar\'e inequality for the unit disc $\mathbb D$ \cite{M}
\begin{equation*}\label{IN2.3}
\inf\limits_{c \in \mathbb R}\left(\iint\limits_{\mathbb D} |f \circ \varphi^{-1}(w)-c|^rdudv\right)^{\frac{1}{r}} \\
\leq B_{r,2}(\mathbb D)
\left(\iint\limits_{\mathbb D} \nabla (f \circ \varphi^{-1}(w))dudv\right)^{\frac{1}{2}}
\end{equation*}
that holds for any $r \geq 1$, we obtain that for any smooth function $f\in L^{1,2}_{A}(\Omega)$

\begin{multline*}
\inf\limits_{c \in \mathbb R}\left(\iint\limits_\Omega |f(z)-c|^rh(z)dxdy\right)^{\frac{1}{r}}
= \inf\limits_{c \in \mathbb R}\left(\iint\limits_\Omega |f(z)-c|^r |J(z,\varphi)| dxdy\right)^{\frac{1}{r}} \\
= \inf\limits_{c \in \mathbb R}\left(\iint\limits_{\mathbb D} |f \circ \varphi^{-1}(w)-c|^rdudv\right)^{\frac{1}{r}}
\leq B_{r,2}(\mathbb D)
\left(\iint\limits_{\mathbb D} \nabla (f \circ \varphi^{-1}(w))dudv\right)^{\frac{1}{2}} \\
= B_{r,2}(\mathbb D)
\left(\iint\limits_{\Omega} \left\langle A(z) \nabla g(z), \nabla f(z) \right\rangle dxdy\right)^{\frac{1}{2}}.
\end{multline*}

Approximating an arbitrary function $f \in W^{1,2}_{A}(\Omega)$ by smooth functions we obtain finally
$$
\inf\limits_{c \in \mathbb R}\left(\iint\limits_\Omega |f(z)-c|^rh(z)dxdy\right)^{\frac{1}{r}} \leq
B_{r,2}(h,A,\Omega) \left(\iint\limits_{\Omega} \left\langle A(z) \nabla f(z), \nabla f(z) \right\rangle dxdy\right)^{\frac{1}{2}},
$$
with the constant
$$
B_{r,2}(h,A,\Omega)=B_{r,2}(\mathbb D)  \leq \left(2^{-1} \pi\right)^{\frac{2-r}{2r}}\left(r+2\right)^{\frac{r+2}{2r}}.
$$

\end{proof}

\section{ Estimates of Sobolev-Poincar\'e constants}

In this section we consider (sharp) upper estimates of Sobolev-Poincar\'e constants in domains that satisfy the quasihyperbolic boundary condition.
Recall that a domain $\Omega$ satisfy the $\gamma$-quasihyperbolic boundary condition \cite{KOT01,KOT02} with some $\gamma>0$, if the growth condition on the quasihyperbolic metric
$$
k_{\Omega}(x_0,x)\leq \frac{1}{\gamma}\log\frac{\dist(x_0,\partial\Omega)}{\dist(x,\partial\Omega)}+C_0
$$
is satisfied for all $x\in\Omega$, where $x_0\in\Omega$ is a fixed base point and $C_0=C_0(x_0)<\infty$.

This quasihyperbolic boundary condition is equivalent to integrability of Jacobians of corresponding quasiconformal mappings with some exponent $\beta>1$. Let us reformulate a theorem about integrability of the Jacobians from \cite{AK} in the convenient for our study form. Firstly, recall that for quasiconformal mappings $\psi:\mathbb D\to\Omega$
the volume derivative
$$
J_{\psi}(w):=\lim\limits_{r\to 0}\frac{|\psi(B(w,r))|}{|B(w,r)|}=|J(w, \psi)|
$$
is defined for almost all $w\in\mathbb D$.

\begin{theorem} \cite{AK} Let $\psi: \mathbb{D} \to \Omega$ be a quasiconformal mapping. Then $J_{\psi} \in L^{\beta}(\mathbb{D})$ for some $\beta>1$ if and only if $\Omega$ satisfy to a $\gamma$-quasihyperbolic boundary conditions for some $\gamma$.
\end{theorem}

Let us remark that the degree of integrability $\beta$ depends only on
$\Omega$ and the quasiconformility coefficient $K(\psi)$.

Our main goal is a reduction of weighted Sobolev-Poincar\'e inequalities to non-weighted embedding theorems. This goal requires the exact value of integrability exponent of Jacobians. It leads us to a following new definition.

Namely, we say that a simply connected domain $\Omega \subset \mathbb C$ is called an $A$-quasiconformal $\beta$-regular domain, $\beta >1$, if
$$
\iint\limits_\mathbb D |J(w, \varphi^{-1})|^{\beta}~dudv < \infty,
$$
where $\varphi: \Omega\to\mathbb D$ is a corresponding $A$-quasiconformal mapping.

Since (see, for example, \cite{Ahl66}) $A$-quasiconformal mappings $\varphi: \Omega\to\mathbb D$ are defined up to conformal automorphisms of $\mathbb D$, a property of quasiconformal $\beta$-regularity doesn't depend on a choice of $\varphi$ and depends on the "quasihyperbolic geometry"  of $\Omega$ only.

Of course, any quasiconformal $\beta$-regular domain satisfies to some $\gamma$-quasihyperbolic boundary conditions and the class of all quasiconformal regular domains coincide with the class of all domains satisfying to the quasihyperbolic boundary conditions.

In \cite{GU14} it was proved that if $\Omega \subset \mathbb C$ is an $A$-quasiconformal $\beta$-regular domain, $\beta >1$, then $\Omega$ has a finite geodesic diameter. Hence, "maze-like" domains \cite{GPU18,KOT02} are not  $A$-quasiconformal $\beta$-regular domains.

Ahlfors domains \cite{Ahl66} (quasidiscs \cite{VGR}) represent an important subclass of $A$-quasi\-con\-for\-mal $\beta$-regular domains. Moreover, in these domains spectral estimates can be specified in terms of the "quasiconformal geometry" of domains (Section~6).

%The domain $\Omega \subset \mathbb C$ is called an $A$-quasiconformal regular domain if it is an $A$-quasiconformal $\beta$-regular domain %for some $\beta>1$.

The following theorem represents  a Sobolev type embedding theorem with estimates of the norm of the embedding operator in quasiconformal regular domains.

\begin{theorem}\label{Th4.3}
Let $A$ belongs to a class  $M^{2 \times 2}(\Omega)$ and $\Omega$ be an $A$-quasiconformal $\beta$-regular domain. Then:
\begin{enumerate}
\item the embedding operator
\[
i_{\Omega}:W^{1,2}_{A}(\Omega) \hookrightarrow L^s(\Omega)
\]
is compact for any $s \geq 1$;
\item for any function $f \in W^{1,2}_{A}(\Omega)$ and for any $s \geq 1$, the Sobolev-Poincar\'e inequality
\[
\inf\limits_{c \in \mathbb R}\|f-c\mid L^s(\Omega)\| \leq B_{s,2}(A,\Omega)
\|f\mid L^{1,2}_{A}(\Omega)\|
\]
holds with the constant
$$
B_{s,2}(A,\Omega) \leq B_{\frac{\beta s}{\beta-1},2}(\mathbb D) \|J_{\varphi^{-1}}\mid L^{\beta}(\mathbb D)\|^{\frac{1}{s}},
$$
where $J_{\varphi^{-1}}$ is a Jacobian of the quasiconformal mapping $\varphi^{-1}:\mathbb D\to\Omega$.
\end{enumerate}
\end{theorem}

\begin{proof}
Let $s \geq 1$. Since $\Omega$ is an $A$-quasiconformal $\beta$-regular domain,
then there exists an $A$-quasiconformal mapping $\varphi: \Omega \to \mathbb D$ satisfies the condition of $\beta$-regularity:
\[
\iint\limits_\mathbb D \big|J(w,\varphi^{-1})\big|^{\beta}~dudv  < \infty.
\]
Hence \cite{VU04} the composition operator for Lebesgue spaces
\[
\varphi^{*}:L^r(\mathbb D) \to L^s(\Omega)
\]
is bounded for $r/(r-s)=\beta$ i.e. for $r=\beta s/(\beta -1)$.

By Theorem~\ref{IsomSS} $A$-quasiconformal mappings $\varphi: \Omega \to \mathbb D$ generate a bounded composition operator on seminormed Sobolev spaces
\[
(\varphi^{-1})^{*}:L^{1,2}_{A}(\Omega) \to L^{1,2}(\mathbb D).
\]
Because the matrix $A$ satisfies to the uniform ellipticity condition~\eqref{UEC} then the norm of Sobolev space $W^{1,2}_{A}(\Omega)$ is equivalent to the norm of Sobolev space $W^{1,2}(\Omega)$ and by \cite{GPU19} we obtain that the composition operator on normed Sobolev spaces
\[
(\varphi^{-1})^{*}:W^{1,2}_{A}(\Omega) \to W^{1,2}(\mathbb D),\,\,\, (\varphi^{-1})^{*}(f)=f \circ \varphi^{-1} ,
\]
is bounded.

Therefore according to the "transfer" diagram \cite{GG94} we obtain that the embedding operator
\[
i_{\Omega}:W^{1,2}_{A}(\Omega) \hookrightarrow L^s(\Omega)
\]
is compact as a composition of three operators: the bounded composition operator on Sobolev spaces
$(\varphi^{-1})^{*}:W^{1,2}_{A}(\Omega) \to W^{1,2}(\mathbb D)$, the compact embedding operator
\[
i_{\mathbb D}:W^{1,2}(\mathbb D) \hookrightarrow L^r(\mathbb D)
\]
and the bounded composition operator on Lebesgue spaces $\varphi^{*}:L^r(\mathbb D) \to L^s(\Omega)$.

Let $s=\frac{\beta -1}{\beta}r$ then by \cite{GPU19} the inequality
\begin{equation}
\label{Weight}
||f\,|\,L^s(\Omega)|| \leq \left(\iint\limits_\mathbb D \big|J(w,\varphi^{-1})\big|^{\beta}~dudv \right)^{{\frac{1}{\beta}} \cdot \frac{1}{s}} ||f\,|\,L^r(\Omega,h)||
\end{equation}
holds for any function $f\in  L^{r}(\Omega,h)$.

Using Theorem \ref{Th4.1} and inequality~\eqref{Weight} we have
\begin{multline*}
\inf_{c \in \mathbb R} \left(\iint\limits_{\Omega} |f(z)-c|^s dxdy\right)^{\frac{1}{s}} \\
{} \leq \left(\iint\limits_\mathbb D \big|J(w,\varphi^{-1})\big|^{\beta}dudv \right)^{{\frac{1}{\beta}} \cdot \frac{1}{s}}
\inf_{c \in \mathbb R} \left(\iint\limits_{\Omega} |f(z)-c|^r h(z)dxdy\right)^{\frac{1}{r}} \\
{} \leq B_{r,2}(\mathbb D)
\left(\iint\limits_\mathbb D \big|J(w,\varphi^{-1})\big|^{\beta}dudv \right)^{{\frac{1}{\beta}} \cdot \frac{1}{s}}
\left(\iint\limits_\Omega \left\langle A(z) \nabla f(z), \nabla f(z) \right\rangle dxdy\right)^{\frac{1}{2}}
\end{multline*}
for $s\geq 1$.
\end{proof}

The following theorem gives compactness of the embedding operator in the limit case $\beta = \infty$:

\begin{theorem}\label{T4.5}
Let $A$ belongs to a class  $M^{2 \times 2}(\Omega)$ and $\Omega$ be an $A$-quasiconformal $\infty$-regular domain. Then:
\begin{enumerate}
\item The embedding operator
\[
i_{\Omega}:W^{1,2}_{A}(\Omega) \hookrightarrow L^2(\Omega),
\]
is compact.

\item For any function $f \in W^{1,2}_{A}(\Omega)$, the Poincar\'e--Sobolev inequality
\[
\inf\limits_{c \in \mathbb R}\|f-c\mid L^2(\Omega)\| \leq B_{2,2}(A,\Omega)
\|f\mid L^{1,2}_{A}(\Omega)\|
\]
holds with the constant $B_{2,2}(A,\Omega) \leq B_{2,2}(\mathbb D) \big\|J_{\varphi^{-1}}\mid L^{\infty}(\mathbb D)\big\|^{\frac{1}{2}}$,
where $J_{\varphi^{-1}}$ is a Jacobian of the quasiconformal mapping $\varphi^{-1}:\mathbb D\to\Omega$.
\end{enumerate}
\end{theorem}

\begin{remark}
The constant $B_{2,2}^2(\mathbb D)=1/\mu_1(\mathbb D)$, where $\mu_1(\mathbb D)=j'^2_{1,1}$ is the first non-trivial Neumann eigenvalue of Laplacian in the unit disc $\mathbb D\subset\mathbb C$.
\end{remark}

\begin{proof}
Since $\Omega$ is an $A$-quasiconformal $\infty$-regular domain,
then there exists an $A$-quasiconformal mapping $\varphi: \Omega \to \mathbb D$ that
generates a bounded composition operator
\[
(\varphi^{-1})^{*}:L^{1,2}_{A}(\Omega) \to L^{1,2}(\mathbb D).
\]
Using  the embedding $L^{1,2}(\mathbb D)\subset L^2(\mathbb D)$ (see, for example, \cite{M}) we obtain that the composition operator on normed Sobolev spaces
\[
(\varphi^{-1})^{*}:W^{1,2}_{A}(\Omega) \to W^{1,2}(\mathbb D)
\]
is bounded also.

Because $\Omega$ is an $A$-quasiconformal $\infty$-regular domain, then the  $A$-quasiconformal mapping $\varphi: \Omega \to \mathbb D$
satisfies the following condition:
\[
\big\|J_{\varphi^{-1}}\mid L^{\infty}(\mathbb D)\big\|=\esssup\limits_{|w|<1}|J(w,\varphi^{-1})|<\infty,
\]
and we have that the composition operator
$$
\varphi^{*}:L^2(\mathbb D) \to L^2(\Omega)
$$
is bounded \cite{VU04}.

Finally, note that in the unit disc $\mathbb D$ the embedding operator
$$
i_{\mathbb D}:W^{1,2}(\mathbb D) \hookrightarrow L^2(\mathbb D)
$$
is compact (see, for example, \cite{M}).
Therefore the embedding operator
$$
i_{\Omega}:W^{1,2}_A(\Omega)\to L_2(\Omega)
$$
is compact as a composition of bounded composition operators $\varphi^{*}$, $(\varphi^{-1})^{*}$ and the compact embedding operator $i_{\mathbb D}$.

Let a  function $f \in L^{2}(\Omega)$. Because quasiconformal mappings possess the Luzin $N$-property, then
$|J(z,\varphi)|^{-1}=|J(w,\varphi^{-1})|$ for almost all $z\in\Omega$ and for almost all $w=\varphi(z)\in \mathbb D$. Hence
the following inequality is correct:
\begin{multline*}
\inf\limits_{c \in \mathbb R}\left(\iint\limits_\Omega |f(z)-c|^2dxdy\right)^{\frac{1}{2}} =
\inf\limits_{c \in \mathbb R}\left(\iint\limits_\Omega |f(z)-c|^2 |J(z,\varphi)|^{-1} |J(z,\varphi)|~dxdy\right)^{\frac{1}{2}} \\
{} \leq \big\|J_{\varphi}\mid L^{\infty}(\Omega)\big\|^{-\frac{1}{2}}
\inf\limits_{c \in \mathbb R}\left(\iint\limits_\Omega |f(z)-c|^2 |J(z,\varphi)|~dxdy\right)^{\frac{1}{2}}.
\end{multline*}

By Theorem~\ref{Th4.1} we obtain

\begin{multline*}
\inf\limits_{c \in \mathbb R}\left(\iint\limits_\Omega |f(z)-c|^2dxdy\right)^{\frac{1}{2}}
\leq \big\|J_{\varphi^{-1}}\mid L^{\infty}(\mathbb D)\big\|^{\frac{1}{2}}
\inf\limits_{c \in \mathbb R}\left(\iint\limits_\mathbb D |g(w)-c|^2~dudv\right)^{\frac{1}{2}} \\
{} \leq B_{2,2}(\mathbb D) \big\|J_{\varphi^{-1}}\mid L^{\infty}(\mathbb D)\big\|^{\frac{1}{2}}
\left(\iint\limits_\Omega \left\langle A(z) \nabla f(z), \nabla f(z) \right\rangle dxdy\right)^{\frac{1}{2}},
\end{multline*}
for any $f\in L_A^{1,2}(\Omega)$.

\end{proof}

\section{Eigenvalue Problem for Neumann Divergence Form Elliptic Operator}

We consider the weak formulation of the Neumann eigenvalue problem~(\ref{EllDivOper}):
\begin{equation}\label{WFWEP}
\iint\limits_\Omega \left\langle A(z )\nabla f(z), \nabla \overline{g(z)} \right\rangle dxdy
= \mu \iint\limits_\Omega f(z)\overline{g(z)}~dxdy, \,\,\, \forall g\in W_{A}^{1,2}(\Omega).
\end{equation}

By the Min--Max Principle (see, for example, \cite{Henr}) the first non-trivial Neumann eigenvalue $\mu_1(\Omega)$ of the divergence form elliptic operator $L_{A}=-\textrm{div} [A(z) \nabla f(z)]$ can be characterized as
$$
\mu_1(A,\Omega)=\min\left\{\frac{\|f \mid L^{1,2}_{A}(\Omega)\|^2}{\|f \mid L^{2}(\Omega)\|^2}:
f \in W^{1,2}_{A}(\Omega) \setminus \{0\},\,\, \iint\limits _{\Omega}f\, dxdy=0 \right\}.
$$

Hence $\mu_1(A,\Omega)^{-\frac{1}{2}}$ is the best constant $B_{2,2}(A,\Omega)$ in the following Poincar\'e inequality
$$
\inf\limits _{c \in \mathbb R} \|f-c \mid L^2(\Omega)\| \leq B_{2,2}(A,\Omega) \|f \mid L^{1,2}_{A}(\Omega)\|, \quad f \in W^{1,2}_{A}(\Omega).
$$

\begin{theorem}\label{Th5.1}
Let $A$ belongs to a class  $M^{2 \times 2}(\Omega)$ and $\Omega$ be an $A$-quasiconformal $\beta$-regular domain. Then the spectrum of the Neumann divergence form elliptic operator $L_{A}$ in $\Omega$ is discrete,
and can be written in the form of a non-decreasing sequence:
\[
0=\mu_0(A,\Omega)<\mu_1(A,\Omega)\leq \mu_2(A,\Omega)\leq \ldots \leq \mu_n(A,\Omega)\leq \ldots ,
\]
and
\[
\frac{1}{\mu_1(A,\Omega)} \leq B_{\frac{2\beta}{\beta -1},2}(\mathbb D)
\|J_{\varphi^{-1}}\mid L^{\beta}(\mathbb D)\|
{} \leq
\frac{4}{\sqrt[\beta]{\pi}} \left(\frac{2\beta -1}{\beta -1}\right)^{\frac{2 \beta-1}{\beta}} \big\|J_{\varphi^{-1}}\mid L^{\beta}(\mathbb D)\big\|,
\]
where $J_{\varphi^{-1}}$ is a Jacobian of the quasiconformal mapping $\varphi^{-1}:\mathbb D\to\Omega$.
\end{theorem}

\begin{proof}
By Theorem~\ref{Th4.3} in the case $s=2$, the embedding operator
$$
i_{\Omega}:W^{1,2}_{A}(\Omega)\to L_2(\Omega)
$$
is compact. Hence the spectrum of the  Neumann divergence form elliptic operator $L_{A}$ is discrete and can be written in the form of a non-decreasing sequence
\[
0=\mu_0(A,\Omega)<\mu_1(A,\Omega)\leq \mu_2(A,\Omega)\leq \ldots \leq \mu_n(A,\Omega)\leq \ldots .
\]

By the Min-Max principle and Theorem~\ref{Th4.3} we have
\[
\inf_{c \in \mathbb R} \left(\iint\limits_{\Omega} |f(z)-c|^2 dxdy\right) \leq B^2_{2,2}(A,\Omega)
\iint\limits_\Omega \left\langle A(z) \nabla f(z), \nabla f(z) \right\rangle dxdy,
\]
where
\[
B_{2,2}(A,\Omega) \leq B_{r,2}(\mathbb D)
\left(\iint\limits_\mathbb D \big|J(w,\varphi^{-1})\big|^{\beta}~dudv \right)^{{\frac{1}{2\beta}}}.
\]

Hence
\[
\frac{1}{\mu_1(A,\Omega)} \leq B^2_{r,2}(\mathbb D)
\left(\iint\limits_\mathbb D \big|J(w,\varphi^{-1})\big|^{\beta}~dudv \right)^{{\frac{1}{\beta}}}.
\]

Using the upper estimate of the $(r,2)$-Poincar\'e constant in the unit disc (see, for example, \cite{GT77,GU16})
\[
B_{r,2}(\mathbb D) \leq \left(2^{-1} \pi\right)^{\frac{2-r}{2r}}\left(r+2\right)^{\frac{r+2}{2r}},
\]
where by Theorem~\ref{Th4.3}, $r=2\beta /(\beta -1)$,
%In this case
%\[
%B_{\frac{2\beta}{\beta -1},2}(\mathbb D) \leq 2\pi^{-\frac{1}{2\beta}} \left(\frac{2\beta -1}{\beta -1}\right)^{\frac{2\beta -1}{2\beta}}.
%\]
we obtain
\[
\frac{1}{\mu_1(A,\Omega)} \leq
\frac{4}{\sqrt[\beta]{\pi}} \left(\frac{2\beta -1}{\beta -1}\right)^{\frac{2 \beta-1}{\beta}} \big\|J_{\varphi^{-1}}\mid L^{\beta}(\mathbb D)\big\|.
\]
\end{proof}

In the case of $A$-quasiconformal $\infty$-regular domains we have:

\begin{theorem}\label{T4.7}
Let $A$ belongs to a class  $M^{2 \times 2}(\Omega)$ and $\Omega$ be an $A$-quasiconformal $\infty$-regular domain. Then the spectrum of the Neumann divergence form elliptic operator $L_{A}$ in $\Omega$ is discrete,
and can be written in the form of a non-decreasing sequence:
\[
0=\mu_0(A,\Omega)<\mu_1(A,\Omega)\leq \mu_2(A,\Omega)\leq \ldots \leq \mu_n(A,\Omega)\leq \ldots ,
\]
and
\begin{equation}
\frac{1}{\mu_1(A,\Omega)} \leq B^2_{2,2}(\mathbb D) \big\|J_{\varphi^{-1}}\mid L^{\infty}(\mathbb D)\big\|
= \frac{\big\|J_{\varphi^{-1}}\mid L^{\infty}(\mathbb D)\big\|}{(j'_{1,1})^2},
\end{equation}
where $j'_{1,1}\approx 1.84118$ denotes the first positive zero of the derivative of the Bessel function $J_1$, and
$J_{\varphi^{-1}}$ is a Jacobian of the quasiconformal mapping $\varphi^{-1}:\mathbb D\to\Omega$.
\end{theorem}

%\section{Examples}

As an applications of Theorem~\ref{T4.7} we consider some examples.

\begin{example}
\label{example1}
The homeomorphism
\[
\varphi(z)= \frac{a}{a^2-b^2}z- \frac{b}{a^2-b^2} \overline{z}, \quad z=x+iy, \quad a>b\geq 0,
\]
is an $A$-quasiconformal and maps the interior of ellipse
$$
\Omega_e= \left\{(x,y) \in \mathbb R^2: \frac{x^2}{(a+b)^2}+\frac{y^2}{(a-b)^2}=1\right\}
$$
onto the unit disc $\mathbb D.$ The mapping $\varphi$ satisfies the Beltrami equation with
\[
\mu(z)=\frac{\varphi_{\overline{z}}}{\varphi_{z}}=-\frac{b}{a}
\]
and the Jacobian $J(z,\varphi)=|\varphi_{z}|^2-|\varphi_{\overline{z}}|^2=1/(a^2-b^2)$.
It is easy to verify that $\mu$ induces, by formula \eqref{Matrix-F}, the matrix function $A(z)$ form
$$
A(z)=\begin{pmatrix} \frac{a+b}{a-b} & 0 \\ 0 &  \frac{a-b}{a+b} \end{pmatrix}.
$$
Given that $|J(w,\varphi^{-1})|=|J(z,\varphi)|^{-1}=a^2-b^2$. Then by Theorem~\ref{T4.7} we have
$$
\frac{1}{\mu_1(A,\Omega_e)} \leq
\frac{1}{(j'_{1,1})^2} \esssup\limits_{|w|<1}|J(w,\varphi^{-1})| = \frac{a^2-b^2}{(j'_{1,1})^2}.
$$

The classical estimate~\ref{PW} with the uniform ellipticity condition states that
$$
\mu_1(A,\Omega_e) \geq \frac{\pi^2}{4(a+b)^2} \frac{a-b}{a+b}
$$
and we have that
$$
\frac{\pi^2}{4(a+b)^2} \frac{a-b}{a+b}< \frac{(j'_{1,1})^2}{a^2-b^2}.
$$
\end{example}

\begin{example}
The homeomorphism
\[
\varphi(z)= \frac{z^{\frac{3}{2}}}{\sqrt{2} \cdot \overline{z}^{\frac{1}{2}}}-1,\,\, \varphi(0)=-1, \quad z=x+iy,
\]
is an $A$-quasiconformal and maps the interior of the ``rose petal"
$$
\Omega_p:=\left\{(\rho, \theta) \in \mathbb R^2:\rho=2\sqrt{2}\cos(2 \theta), \quad -\frac{\pi}{4} \leq \theta \leq \frac{\pi}{4}\right\}
$$
onto the unit disc $\mathbb D$.
The mapping $\varphi$ satisfies the Beltrami equation with
\[
\mu(z)=\frac{\varphi_{\overline{z}}}{\varphi_{z}}=-\frac{1}{3}\frac{z}{\overline{z}}
\]
and the Jacobian $J(z,\varphi)=|\varphi_{z}|^2-|\varphi_{\overline{z}}|^2=1$.
We see that $\mu$ induces, by formula \eqref{Matrix-F}, the matrix function $A(z)$ form
$$
A(z)=\begin{pmatrix} \frac{|3\overline{z}+z|^2}{8|\overline{z}|^2} & \frac{3}{4}\Imag \frac{z}{\overline{z}} \\ \frac{3}{4}\Imag \frac{z}{\overline{z}} & \frac{|3\overline{z}-z|^2}{8|\overline{z}|^2} \end{pmatrix}.
$$
Given that $|J(w,\varphi^{-1})|=|J(z,\varphi)|^{-1}=1$. Then by Theorem~\ref{T4.7} we have
$$
\frac{1}{\mu_1(A,\Omega_p)} \leq
\frac{1}{(j'_{1,1})^2} \esssup\limits_{|w|<1}|J(w,\varphi^{-1})| = \frac{1}{(j'_{1,1})^2}.
$$

The classical estimate~\ref{PW} with the uniform ellipticity condition states that
$$
\mu_1(A,\Omega_p) \geq \left(\frac{\pi}{4}\right)^2
$$
and we have that
$$
\left(\frac{\pi}{4}\right)^2<(j'_{1,1})^2 \quad \text{or} \quad \frac{\pi}{4}<j'_{1,1}.
$$

\end{example}

\begin{example}
The homeomorphism
\[
\varphi(z)= \frac{2 \cdot z^{\frac{3}{8}}}{\overline{z}^{\frac{1}{8}}}-1,\,\, \varphi(0)=-1, \quad z=x+iy,
\]
is an $A$-quasiconformal and maps the interior of the non-convex domain
$$
\Omega_c:=\left\{(\rho, \theta) \in \mathbb R^2:\rho=\cos^{4}\left(\frac{\theta}{2}\right), \quad - \pi \leq \theta \leq \pi\right\}
$$
onto the unit disc $\mathbb D$.
The mapping $\varphi$ satisfies the Beltrami equation with
\[
\mu(z)=\frac{\varphi_{\overline{z}}}{\varphi_{z}}=-\frac{1}{3}\frac{z}{\overline{z}}
\]
and the Jacobian
$$J(z,\varphi)=|\varphi_{z}|^2-|\varphi_{\overline{z}}|^2=\frac{1}{2\cdot |z|^{\frac{3}{2}}}.
$$
We see that $\mu$ induces, by formula \eqref{Matrix-F}, the matrix function $A(z)$ form
$$
A(z)=\begin{pmatrix} \frac{|3\overline{z}+z|^2}{8|\overline{z}|^2} & \frac{3}{4}\Imag \frac{z}{\overline{z}} \\ \frac{3}{4}\Imag \frac{z}{\overline{z}} & \frac{|3\overline{z}-z|^2}{8|\overline{z}|^2} \end{pmatrix}.
$$
Given that $|J(w,\varphi^{-1})|=|J(z,\varphi)|^{-1}=2\cdot |z|^{\frac{3}{2}}$. Then by Theorem~\ref{T4.7} we have
$$
\frac{1}{\mu_1(A,\Omega_c)} \leq
\frac{1}{(j'_{1,1})^2} \esssup\limits_{|w|<1}|J(w,\varphi^{-1})| \leq \frac{2}{(j'_{1,1})^2}.
$$
\end{example}

\section{Spectral estimates in quasidiscs}

In this section we precise Theorem~\ref{Th5.1} for Ahlfors-type domains (i.e. quasidiscs) using the weak inverse H\"older inequality and the sharp estimates of the constants in doubling conditions for measures generated by Jacobians of quasiconformal mappings  \cite{GPU17_2}.

Recall that a domain $\Omega$ is called a $K$-quasidisc if it is the image of the unit disc $\mathbb D$ under a $K$-quasicon\-for\-mal homeomorphism of the plane onto itself. A domain $\Omega$ is a quasidisc if it is a $K$-quasidisc for some $K \geq 1$.

According to \cite{GH01}, the boundary of any $K$-quasidisc $\Omega$
admits a $K^{2}$-quasi\-con\-for\-mal reflection and thus, for example,
any quasiconformal homeomorphism $\psi:\mathbb{D}\to\Omega$ can be
extended to a $K^{2}$-quasiconformal homeomorphism of the whole plane
to itself.

Recall that for any planar $K$-quasiconformal homeomorphism $\psi:\Omega\rightarrow \Omega'$
the following sharp result is known: $J(w,\psi)\in L^p_{\loc}(\Omega)$
for any $1 \leq p<\frac{K}{K-1}$ (\cite{Ast,G81}).

In \cite{GPU17_2} was proved but not formulated the result concerning an estimate of the constant in the inverse H\"older inequality for Jacobians of quasiconformal mappings.

\vskip 0.2cm

\begin{theorem}
\label{thm:IHIN}
Let $\psi:\mathbb R^2 \to \mathbb R^2$ be a $K$-quasiconformal mapping. Then for every disc $\mathbb D \subset \mathbb R^2$ and
for any $1<\kappa<\frac{K}{K-1}$ the inverse H\"older inequality
\begin{equation*}\label{RHJ}
\left(\iint\limits_{\mathbb D} |J(w,\psi)|^{\kappa}~dudv \right)^{\frac{1}{\kappa}}
\leq \frac{C_\kappa^2 K \pi^{\frac{1}{\kappa}-1}}{4}
\exp\left\{{\frac{K \pi^2(2+ \pi^2)^2}{2\log3}}\right\}\iint\limits_{\mathbb D} |J(w,\psi)|~dudv
\end{equation*}
holds. Here
$$
C_\kappa=\frac{10^{6}}{[(2\kappa -1)(1- \nu)]^{1/2\kappa}}, \quad \nu = 10^{8 \kappa}\frac{2\kappa -2}{2\kappa -1}(24\pi^2K)^{2\kappa}<1.
$$
\end{theorem}

If $\Omega$ is a $K$-quasidisc, then given the previous theorem and that a quasiconformal mapping $\psi:\mathbb{D}\to\Omega$ allows $K^2$-quasiconformal reflection \cite{Ahl66, GH01}, we obtain the following assertion.

\begin{corollary}\label{Est_Der}
Let $\Omega\subset\mathbb R^2$ be a $K$-quasidisc and $\varphi:\Omega \to \mathbb D$ be an $A$-quasiconformal mapping. Assume that  $1<\kappa<\frac{K}{K-1}$.
Then
\begin{equation*}\label{Ineq_2}
\left(\iint\limits_{\mathbb D} |J(w,\varphi^{-1})|^{\kappa}~dudv \right)^{\frac{1}{\kappa}}
\leq \frac{C_\kappa^2 K^2 \pi^{\frac{1}{\kappa}-1}}{4}
\exp\left\{{\frac{K^2 \pi^2(2+ \pi^2)^2}{2\log3}}\right\}\cdot |\Omega|.
\end{equation*}
where
$$
C_\kappa=\frac{10^{6}}{[(2\kappa -1)(1- \nu)]^{1/2\kappa}}, \quad \nu = 10^{8 \kappa}\frac{2\kappa -2}{2\kappa -1}(24\pi^2K^2)^{2\kappa}<1.
$$
\end{corollary}

Combining Theorem~\ref{Th5.1} and Corollary~\ref{Est_Der} we obtain
spectral estimates of linear elliptic operators in divergence form
with Neumann boundary conditions in Ahlfors-type domains.

\begin{theorem}\label{Quasidisk}
Let $\Omega$ be a $K$-quasidisc. Then
\begin{equation*}
\mu_1(A,\Omega) \geq \frac{M(K)}{|\Omega|}=\frac{M^{*}(K)}{R^2_{*}},
\end{equation*}
where $R_{*}$ is a radius of a disc $\Omega^{*}$ of the same area as $\Omega$ and
$M^{*}(K)=M(K)\pi^{-1}$.
\end{theorem}
The quantity $M(K)$ depends only on a quasiconformality
coefficient K of $\Omega$:
\[
M(K):= \frac{\pi}{K^2}
\exp\left\{{-\frac{K^2 \pi^2(2+ \pi^2)^2}{2\log3}}\right\}
\inf\limits_{1< \beta <\beta^{*}}
\Biggl\{
\left(\frac{2\beta -1}{\beta -1}\right)^{-\frac{2 \beta-1}{\beta}} C^{-2}_{\beta}
\Biggr\},
\]
\[
C_\beta=\frac{10^{6}}{[(2\beta -1)(1- \nu(\beta))]^{1/2\beta}},
\]
where $\beta^{*}=\min{\left(\frac{K}{K-1}, \widetilde{\beta}\right)}$, and $\widetilde{\beta}$ is the unique solution of the equation
$$\nu(\beta):=10^{8 \beta}\frac{2\beta -2}{2\beta -1}(24\pi^2K^2)^{2\beta}=1.
$$
The function $\nu(\beta)$ is a monotone increasing function. Hence for
any $\beta < \beta^{*}$ the number $(1- \nu(\beta))>0$ and $C_\beta > 0$.

\begin{proof}
Given that, for $K\geq 1$, $K$-quasidiscs are $A$-quasiconformal $\beta$-regular domains if $1<\beta<\frac{K}{K-1}$. Therefore, by Theorem~\ref{Th5.1} for $1<\beta<\frac{K}{K-1}$ we have
\begin{equation}\label{Inequal_1}
\frac{1}{\mu_1(A,\Omega)} \leq
\frac{4}{\sqrt[\beta]{\pi}} \left(\frac{2\beta -1}{\beta -1}\right)^{\frac{2 \beta-1}{\beta}} \big\|J_{\varphi^{-1}}\mid L^{\beta}(\mathbb D)\big\|.
\end{equation}
Now, using Corollary~\ref{Est_Der} we estimate the quantity $\|J_{\varphi^{-1}}\,|\,L^{\beta}(\mathbb D)\|$.
Direct calculations yield
\begin{multline}\label{Inequal_2}
\|J_{\varphi^{-1}}\,|\,L^{\beta}(\mathbb D)\| =
\left(\iint\limits_{\mathbb D} |J(w,\varphi^{-1})|^{\beta}~dudv \right)^{\frac{1}{\beta}} \\
\leq \frac{C^2_{\beta} K^2 \pi^{\frac{1-\beta}{\beta}}}{4} \exp\left\{{\frac{K^2 \pi^2(2+ \pi^2)^2}{2\log3}}\right\} \cdot |\Omega|.
\end{multline}
Finally, combining inequality \eqref{Inequal_1} with inequality \eqref{Inequal_2} after some computations, we obtain
\[
\frac{1}{\mu_1(A,\Omega)} \leq
\frac{C^2_{\beta} K^2}{\pi} \left(\frac{2\beta -1}{\beta -1}\right)^{\frac{2 \beta-1}{\beta}}  \exp\left\{{\frac{K^2 \pi^2(2+ \pi^2)^2}{2\log3}}\right\} \cdot |\Omega|.
\]
\end{proof}

%$\mathbf{Remark^*.}$
Let $\varphi:\Omega \to \Omega'$ be quasiconformal mappings. We note that there exist so-called volume-preserving maps, i.e. $|J(z,\varphi)|=1$, $z \in \Omega$. Examples of such maps were considered in the previous section.
Now we construct another examples of such maps.

Let $f\in L^{\infty}(\mathbb R)$. Then $\varphi(x,y)=(x+f(y),\,y)$ is a quasiconformal mapping with a quasiconformality coefficient $K=\lambda/J_{\varphi}(x,y)$.
Here $\lambda$ is the largest eigenvalue of the matrix $Q=DD^T$, where
$D=D\varphi(x,y)$ is Jacobi matrix of mapping $\varphi=\varphi(x,y)$ and $J_{\varphi}(x,y)=\det D\varphi(x,y)$ is its Jacobian.

It is easy to see that the Jacobi matrix corresponding to the mapping $\varphi=\varphi(x,y)$ has the form
\[
D=\left(\begin{array}{cc}
1 & f'(y)\\
0 & 1
\end{array}\right).
\]

A basic calculation implies $J_{\varphi}(x,y)=1$ and
\[
\lambda=\left(1+\frac{\left(f'(y)\right)^{2}}{2}\right)\left(1+\sqrt{1-\frac{4}{\left(2+\left(f'(y)\right)^{2}\right)^{2}}}\right)\,.
\]

Therefore any mapping $\varphi=\varphi(x,y)$ is a quasiconformal mapping
from $\mathbb R^{2}\to \mathbb R^{2}$ with $J_{\varphi}(x,y)=1$ and arbitrary large quasiconformality coefficient.

We can use their restrictions $\varphi|_{\mathbb D}$ to the unit disc $\mathbb D$. Images
can be very exotic quasidiscs.

If $a>0$ then mappings $\varphi(x,y)=(ax+f(y),\,\frac{1}{a}y)$ have
similar properties.

In this case we obtain lower estimates of the first non-trivial Neumann eigenvalues of the divergent form elliptic operator $L_A$ in $A$-quasiconformal $\beta$-regular domains via the Sobolev-Poincar\'e constant for the unit disc $\mathbb D$.

\vskip 0.2cm

\textbf{Acknowledgements.} The first author was supported by the United States-Israel Binational Science Foundation (BSF Grant No. 2014055).

\vskip 0.3cm

Department of Mathematics, Ben-Gurion University of the Negev, P.O.Box 653, Beer Sheva, 8410501, Israel

\emph{E-mail address:} \email{vladimir@math.bgu.ac.il} \\

Division for Mathematics and Computer Sciences, Tomsk Polytechnic University,
634050 Tomsk, Lenin Ave. 30, Russia; Regional Scientific and Educational Mathematical
Center, Tomsk State University, 634050 Tomsk, Lenin Ave. 36, Russia

%\emph{Current address:} Department of Mathematics, Ben-Gurion University of the Negev, P.O.Box 653,
%Beer Sheva, 8410501, Israel
							
\emph{E-mail address:} \email{vpchelintsev@vtomske.ru}   \\
			
Department of Mathematics, Ben-Gurion University of the Negev, P.O.Box 653, Beer Sheva, 8410501, Israel
							
\emph{E-mail address:} \email{ukhlov@math.bgu.ac.il}

\end{document}